\documentclass[leqno,11pt,letterpaper, english]{amsart}
\usepackage[usenames,dvipsnames]{color}
\usepackage[colorlinks=true,linkcolor=Red,citecolor=Green]{hyperref}
\usepackage{amsmath,amssymb,amsthm,graphicx,url}
\usepackage{epstopdf}
\usepackage{color}
\usepackage{dsfont}
\usepackage[utf8]{inputenc} 
\usepackage[T1]{fontenc}
\newcommand{\xqedhere}[1]{%
    \rlap{%
         \hbox to#1{%
           \hfil
           \llap{%
               \ensuremath{\square}
           }%
       }%
   }%
}

\def\pasdegrille{\let\grille = \pasgrille}

\def\aat#1#2#3{
\divide \dimen1 by 48 \dimen3=\dimen1 \multiply \dimen1 by #1
\advance \dimen1 by -\dimen3 \divide \dimen1 by 101 \multiply
\dimen1 by 100 \divide \dimen2 by \count11 \multiply \dimen2 by #2
\setbox0=\hbox{#3}\ht0=0pt\dp0=0pt
  \rlap{\kern\dimen1 \vbox to0pt{\kern-\dimen2\box0\vss}}\dimen1= \wd1
\dimen2=\ht1}
\def\pasgrille{
\count12= \dimen1 \divide \count12 by 50 \divide \dimen2 by
\count12 \count11 =\dimen2 \ \divide \dimen1 by 48
\setlength{\unitlength}{\dimen1} \smash{\rlap{\ }} \dimen1= \wd1
\dimen2=\ht1 }
\def\grille{
\count12= \dimen1 \divide \count12 by 50 \divide \dimen2 by
\count12 \count11 =\dimen2 \ \divide \dimen1 by 48
\setlength{\unitlength}{\dimen1}
\smash{\rlap{\graphpaper[1](0,0)(50, \count11)}} \dimen1= \wd1
\dimen2=\ht1 }

\pasdegrille

\newcommand{\ud}{\,\mathrm{d}}

\newcommand{\be}{\begin{equation}}
\newcommand{\ee}{\end{equation}}

\theoremstyle{plain}

\newtheorem{thm}{Theorem}

\newtheorem{prop}{Proposition}[section]

\newtheorem{lem}[prop]{Lemma}

\theoremstyle{definition}

\newtheorem{rem}[prop]{Remark}

\numberwithin{equation}{section}

\def\squarebox#1{\hbox to #1{\hfill\vbox to #1{\vfill}}}

\usepackage{amsxtra}

\ifx\pdfoutput\undefined
  \DeclareGraphicsExtensions{.pstex, .eps}
\else
  \ifx\pdfoutput\relax
    \DeclareGraphicsExtensions{.pstex, .eps}
  \else
    \ifnum\pdfoutput>0
      \DeclareGraphicsExtensions{.pdf}
    \else
      \DeclareGraphicsExtensions{.pstex, .eps}
    \fi
  \fi
\fi

\title[Propagation of smallness and spectral estimates]{Propagation of smallness\\ and spectral estimates}

\author[N. Burq]{Nicolas Burq}
\address{Universit{\'e} Paris-Saclay, Math{\'e}matiques, UMR 8628 du CNRS, B{\^a}t 307, 91405  Orsay Cedex, France,   and Institut Universitaire de France}
\email{Nicolas.burq@universite-paris-saclay.fr}
\author[I. Moyano]{Ivan Moyano}
\address{Universit\'e de Nice Sophia-Antipolis
Parc Valrose, Laboratoire J.A. Dieudonn\'e,
UMR 7351 du CNRS
06108 NICE Cedex 02
FRANCE}
\email{Ivan.Moyano@unice.fr}

\usepackage{amssymb}
\usepackage{amsmath, amsthm, amsopn, amsfonts}

\def\11{{\rm 1~\hspace{-1.4ex}l} }
\def\R{\mathbb R}

\def\N{\mathbb N}

\begin{document}

\begin{abstract}
The purpose of this article is to show that the spectral projector estimates for Laplace operators can be deduced from Logunov-Malinnikova's Propagation of smallness estimates for harmonic functions~\cite{LoMa17, Lo18, Lo18-1}. The main point is to pass from the local estimates obtained in~\cite{BuMo19} (on a compact manifold) to {\em global} estimates. We also state classical consequences in terms of observability and control for heat equations, which are direct consequences of these spectral projector estimates..
\begin{center}
 
  \rule{0.6\textwidth}{.4pt} \end{center}\end{abstract}   

\ \vskip -1cm \noindent\hfil\rule{0.9\textwidth}{.4pt}\hfil \vskip 1cm 
 \maketitle   
 \section{Introduction} 
 In this article, we continue our investigation of Logunov and Malinnikova's results on the propagation of smallness for elliptic equations and their consequences on the propagation of smallness, observation and control for heat equations. Here our purpose is to extend our previous results~\cite{BuMo19}  on compact manifolds (with or without boundaries) to the non compact setting.
We consider on $\mathbb{R}^d$ the following Laplace operator
$$ \Delta = \frac 1 {\kappa(x)} \sum_{i,j} \partial_{x_i} g^{i,j}(x) \kappa (x) \partial_{x_j},$$
where we assume that the coefficients $\kappa, g$ are Lipschitz  and that $g$ is uniformly elliptic:
\begin{equation}\label{ellipticity}
 \| \kappa\|_{W^{1, \infty}(\mathbb{R}^d) } + \| g\|_{W^{1, \infty}(\mathbb{R}^d) } \leq A, \qquad \exists a>0;  \forall x \in \mathbb{R}^d, \quad a\text{Id} \leq g(x), \quad \kappa(x) \geq a
 \end{equation}
 Recall that the $n$-Hausdorff content (or measure) of a set $E\subset \mathbb{R}^d$ is 
\begin{equation}\label{hausdorf}
 \mathcal{C}_{\mathcal{H}}^n (E) = \inf \{ \sum_j r_j ^n; E \subset \cup_jB(x_j, r_j)\},
 \end{equation}
and (for a bounded set) the Hausdorff dimension of $E$ is defined as 
$$ \text{dim}_\mathcal{H} (E) = \inf \{n;   \mathcal{C}_{\mathcal{H}}^n(E) =0 \}.$$
Let $\omega \subset \mathbb{R}^d$ satisfying the following condition (see~\cite{BuJo-16,  LRMo16})
\begin{equation}\label{mesure}
\exists R, \delta >0; \quad \forall x \in \mathbb{R}^d, \,\, \text{meas} ( \omega \cap B(x, R)) \geq \delta,
\end{equation}
or 
\begin{equation}\label{mesurebis}
\exists R, \delta >0; \quad \forall x \in \mathbb{R}^d, \, \, \mathcal{C}_{\mathcal{H}}^n ( \omega \cap B(x, R)) \geq \delta.
\end{equation}
The first purpose of this note is to show that Jerison-Lebeau's spectral estimate from~\cite{LeMo19} still holds under these very weak assumptions (assuming in~\eqref{mesurebis} that $n<d$ is sufficiently close to $d$ i.e. the set $\omega$ is not too thin). Let us recall that the operator $-\Delta$ on $L^2( \mathbb{R}^d,  \kappa (x) dx)$, with domain $H^2(\mathbb{R}^d)$ is self adjoint with non negative (continuous) spectrum. As a consequence, using the spectral theorem, it is possible to write (here $dm_\lambda$ is the spectral measure of the operator $\sqrt{ - \Delta}$):
$$ u = \int_0^{+\infty} dm_\lambda u, \qquad \Pi_\mu (u) = 1_{\sqrt{ - \Delta} \leq \mu} u. 
$$ Moreover, if $\phi,\varphi$ are bounded functions on $\R$, we have 
$$   \phi(\sqrt{ - \Delta}) u = \int_0^{+\infty} \phi( \lambda)dm_\lambda u, \quad \bigl( \phi(\sqrt{- \Delta})u, \varphi(\sqrt{ -\Delta})u  \bigr)_{L^2( \mathbb{R}^d, \kappa (x) dx)} = \int_0^{+\infty} \phi( \lambda) \varphi(\lambda) (dm_\lambda u, u) 
$$ 
and consequently 
\begin{equation}\label{bound}
 \| \phi( \sqrt{ - \Delta}) \Pi_\mu u \| _{L^2( \kappa dx)} \leq \sup_{\lambda\in [0, \mu]} | \phi( \lambda)| \| u \| _{L^2(\kappa dx)}.
\end{equation}
\begin{thm}\label{princ}

Let 
$$ \Pi_\mu (u) = 1_{\sqrt{ - \Delta} \leq \mu} u = \int_{\lambda=0}^\mu dm_\lambda  u.
$$ 
Assume that $\omega$ satisfies the assumption~\eqref{mesure}. Then there exists $C>0$ such that for any $u \in L^2( \mathbb{R}^d)$,
\begin{equation}\label{spec1}
\| u \|_{L^2(\mathbb{R}^d; \kappa dx) } \leq C e^{C \mu}  \| u \|_{L^2(\omega; \kappa dx) }.
\end{equation}

Assume now that $\omega$ satisfies the assumption~\eqref{mesurebis} for some $n<d$ sufficiently close to $d$. Then there exists $C>0$ such that for anu $u \in L^2( \mathbb{R}^d)$,
\begin{equation}\label{spec2}
 \| u \|^2_{L^2(\mathbb{R}^d; \kappa dx) } \leq C e^{C \mu}  \sum_{k\in \mathbb{Z^d}} \| u \|^2_{L^\infty(\omega \cap B( k, R)) }.
 \end{equation}
\end{thm}
Let us now give three applications to the observation and control of the heat equation. The first application is very standard and follows from the works by Miller~\cite{Mi10},  Phung-Wang~\cite{PhWa13} and Apraiz, Escauriaza, Wang, Zhang~\cite{AEWZ14,ApEs13}. We also refer to~\cite[Section 5] {BuMo19}.
\begin{thm} [Null controllability from sets of positive measure]\label{control}
Let $F \subset (0,T) $ a set of positive measure. Let $E\subset \mathbb{R}^d$ satisfying~\eqref{mesure}. Then, there exists $C>0$ such that for any $u_0 \in L^2(M)$ the solution 
$u = e^{t \Delta} u_0$ to the heat equation 
$$\partial _t u - \Delta u =0,  \qquad  u \mid_{\left\{ t=0 \right\} } = u_0,
$$ satisfies (recall that $\kappa$ satisfies~\eqref{ellipticity})
\begin{equation}\label{obs} 
\| e^{T \Delta } u_0 \|^2_{L^2(M)} \leq C \int_{F\times E} | u|^2(t,x) \kappa(x)dx dt.
\end{equation}
As a consequence, for all $u_0, v_0 \in L^2(M)$ there exists $f\in L^2(F\times E)$ such that the solution to 
\begin{equation*}
(\partial_t - \Delta ) u = f1_{F} (t,x), \qquad u \mid_{\left\{ t=0 \right\} } = u_0,\\
\end{equation*}
satisfies 
$$ u \mid_{\left\{ t=T \right\} } =e^{T\Delta} v_0.$$
\end{thm}
The second application is a control and observability result from sets of positive Hausdorff content. It follows from~\eqref{spec2} following the proof of \cite[Theorem 3]{BuMo19} given  in~\cite[Section 5]{BuMo19}. The only difference is that in~\cite{BuMo19} we use the duality between continuous functions and Radon measures while here we use the duality between the spaces $G, \widetilde{G}$ defined below. 
\begin{thm} [Null controllability from sets of positive Hausdorff content]\label{controlbis}
Let $F \subset (0,T) $ be a set of positive measure. Let $E\subset \mathbb{R}^d$ satisfying~\eqref{mesurebis}. Then, there exists $C>0$ such that for any $u_0 \in L^2(M)$ the solution 
$u = e^{t \Delta} u_0$ to the heat equation 
$$\partial _t u - \Delta u =0,  \qquad     u \mid_{\left\{ t=0 \right\}} = u_0,
$$ satisfies (recall that $\kappa$ satisfies~\eqref{ellipticity})
\begin{equation}\label{obsbis} 
\| e^{T \Delta } u_0 \|^2_{L^2(M)} \leq C \sum_{k \in \mathbb{Z}^d} \int_F \sup_{ x\in E\cap B(k, R)} | u|^2(t,x)  dt.
\end{equation}
Let us denote by $G$ the set of functions on $F$ with values locally bounded Radon measures on $\overline{E}$, such that 
$$ \| f\|_{G}^2 = \sum_{k\in \mathbb{Z}^d} \int_{F} |f|^2 (\overline{E}\cap B(k,R)) (t) dt < +\infty.
$$
Remark that $G$ is the dual space of $\widetilde{G}$, defined as the space of functions $g$ on $F$ with values continuous functions on $\overline{E}$  such that 
$$ \| g\|^2_{\widetilde{G}}= \sum_{k\in \mathbb{Z}^d} \int_{F} \sup_k |f (E\cap B(k,R))|^2 (t) dt < + \infty.
$$
We endow $G$ with its natural norm. 
We deduce from~\eqref{obsbis} that for all $u_0, v_0 \in L^2(M)$ there exists $f\in G$ such that the solution to 
\begin{equation}\label{eqcont}
\begin{gathered}
(\partial_t - \Delta ) u = f1_{F \times \overline{E}} (t,x), \qquad u \mid_{ \left\{ t=0 \right\} } = u_0,\\
\end{gathered}
\end{equation}
satisfies 
$$ u \mid_{ \left\{t=T\right\}  } =e^{T\Delta} v_0.$$
\end{thm}
\begin{rem} For $\chi \in C^\infty_0 ( \mathbb{R}^d)$, from~\cite[Proposition 2.1]{BuMo19}, for $\sigma >0$ large enough the map 
$$ u \in D( - \Delta)^\sigma \mapsto \chi u  \in C^0 $$ is well defined. We deduce by duality that if $\mathcal{M}$ is the set or Radon measures on $\overline {E}$, the map
$$ v \in \mathcal{M} \mapsto \chi v \in D( - \Delta)^{-\sigma }
$$ 
is also well defined and continuous.
As a consequence, $G$ is continously embeddeed in 
$$ \mathcal{H} = L^2(F;  D( - \Delta)^{-\sigma }), 
$$
and equation~\eqref{eqcont} is just solved by
$$ u = e^{t \Delta} u_0 + \int_0^t e^{(t-s)\Delta} f1_{F \times\overline{E}} (s,x) ds .
$$
\end{rem}

The third application is the observation and control for discrete times from~\cite[Theorem 4]{BuMo19}.
\begin{thm}[Observability and exact controllability using controls localised at fixed times]\label{control-hausbis} 
 Let  $m>0$, $\tau \in (0,1)$ and $D>0$. There exists  $C>0$,  such that if  $E\subset \mathbb{R}^d$ satisfies ~\eqref{mesure} 
then for any sequence $(s_n)_{n\in \mathbb{N}}$,
$$J= \{ 0<\cdots <  s_n<  \cdots  < s_0< T\} $$ converging not too fast to $0$, i.e.,
$$  \exists \tau \in (0,1); \forall n \in \N,  (s_{n} - s_{n+1}) \geq \tau ( s_{n-1} - s_{n}), $$
we have that  for any $u_0 \in L^2(M)$, the solution 
$u = e^{t \Delta} u_0$ to the heat equation 
$$\partial _t u - \Delta u =0,   \qquad u \mid_{\left\{t=0\right\}} = u_0,
$$ satisfies 
\begin{equation}\label{obster} 
\| e^{T \Delta } u_0 \|^2_{L^2(M)} \leq   C \sup_{n \in \N} e^{-\frac D{s_n- s_{n+1}} }\int_{E} |e^{{s_n\Delta} } u_0|^2 (s_n,x) dx.
 \end{equation}

As a consequence, given any sequence $(t_n)_{n\in \mathbb{N}}$,
$$J= \{ 0< t_0 < \cdots< t_n < \cdots< T\} $$ converging not too fast to $T$,
\begin{equation}\label{nottoofast} \exists 0<\tau<1; \forall n \in \N,  (t_{n+1} - t_{n} )\geq \tau ( t_{n} - t_{n-1}),
\end{equation} for all $u_0, v_0 \in L^2(M)$ 
there exists $(f_j) $ a sequence of functions on  $E_1$ such that 
$$ \sum_j e^{\frac D{(t_{j+1}-t_j)} }\|f_j\|_{L^2(E)}  < +\infty, $$ and the solution to 
\begin{equation}\label{solmesurebis}
(\partial_t - \Delta ) u = \sum_{j=1}^{+\infty} \delta_{t=t_j}\otimes f_j (x)1_{E_1},  \quad u \mid_{t=0} = u_0\end{equation}
satisfies 
$$ u \mid_{t> T} =e^{t\Delta} v_0.$$
\end{thm}

\begin{rem} In ~\cite[Theorem 4]{BuMo19} the observation involves the $L^1$ norm rather than the $L^2$ norm and consequently, the controls  $f$,$f_j$ are $L^\infty$. Here, for simplicity we kept $L^2$ in both cases. The proof of Theorem~\ref{control-hausbis} is obtained from the proof of~\cite[Theorem 4]{BuMo19} modulo this simple modification.
\end{rem}
\begin{rem}
We only stated three typical consequences. We could have mixed these examples (control on sets of positive Hausdorff content  {\em and} countable times).
\end{rem}
\begin{rem} In~\cite{BuMo19}, the spectral estimates are obtained in manifolds with boundaries under Dirichlet or Neumann boundary conditions. It would also be possible to include such boundaries in a non compact setting provided we still control the geometry at infinity. For example, we could include holes as long as their diameters remain in a compact set of $(0, +\infty)$ and they remain disjoint  from each other (with a fixed lower bound on their distances).
\end{rem} 
\section{proof of Theorem~\ref{princ}}

We start with the first part in Theorem~\ref{princ}. 
Let us define
\begin{equation}\label{prolong}
 v_\mu(t,x) =  \int_{\lambda=0}^\mu \frac{ \sinh(\lambda t)} {\lambda}  dm_\lambda u.
\end{equation}
Recall that 
$$ {- \Delta} v_\mu= \int_{\lambda=0}^\mu \frac{ \sinh(\lambda t)} {\lambda}  \lambda^2 dm_\lambda u.
$$
As a consequence, we have  
$$ \Bigl( \frac 1 {\kappa} \text{div} g^{-1} \kappa \nabla_x + \partial_t^2\Bigr) v_{\mu} =0 \Rightarrow   \Bigl(\text{div} g^{-1} \kappa \nabla_x + \partial_t\kappa \partial_t \Bigr) v_{\mu}=0. $$ 
Without loss of generality we can assume that in~\eqref{mesure} and~\eqref{mesurebis}, the constant $R$ is chosen large enough so that 
$$ \cup_{p \in \mathbb{Z}^d} B(p, R) = \mathbb{R}^d.$$ Let $1= \sum_p \chi(x-p)$ be a partition of unity associated. 
 Consider, for $T_2 > T_1 > 0$ the sets 
\begin{equation*}
\mathcal{K}_p  := [-T_1, T_1] \times \overline{B(p, R)}, \quad \Omega_p  := (-T_2, T_2) \times B(p, 2R),   \quad F_p = \omega \cap B(p,R), \quad  E_p= \{0\} \times F_p
\end{equation*} which by construction satisfy the inclusions $E_p \subset \mathcal{K}_p\subset \Omega_p$. 
Now assumption~\eqref{mesure} ensures that $|F_p|\geq \delta $. Remark that in~\eqref{hausdorf}, since $F_p$ has a diameter at most $2R$, when computing the Hausdorff content of $F_p$, we can assume that $\forall j, r_j\leq 4R$. As a consequence,  
$$\mathcal{C}^{d-1 - \epsilon} (E_p)\geq (4R)^{- \epsilon}\mathcal{C}^{d-1} (E_p) \geq c \delta.$$
For sufficiently small $\epsilon >0$ we can now apply~\cite[Theorem 5.1]{LoMa17} and get
\begin{equation}\label{log-mal} 
\sup_{\mathcal{K}_p} |\nabla_{t,x} v_\mu| \leq C \bigl( \sup_{E_p} |\nabla_{t,x} v_\mu| \bigr)^{\alpha} \bigl( \sup_{\Omega_p} |\nabla_{t,x} v_\mu|\bigr)^{1- \alpha},
\end{equation}
where the constant is uniform with respect to $p$  because the metric and $\kappa$ are bounded on $B(p, 2R)$ (uniformly with respect to $p$  while the metric is uniformly elliptic). 

We first replace $L^\infty$ norms by $L^2$ norms. We can assume $v_\mu\not \equiv 0$
Let 
\begin{align*}
a &= \left(  \epsilon \frac{\sup_{\mathcal{K}_p} |\nabla_{t,x} v_{\mu}|}{ \bigl( \sup_{\Omega_p} |\nabla_{t,x} v_{\mu}|\bigr)^{1- \alpha}} \right)^{\frac 1 \alpha}, \\
E_p' &= \{ x \in E_p; | v_\mu(x) | \leq a\}.
\end{align*} Then for $\epsilon C <1$, $|E'_p| \leq \delta /2$. Indeed, otherwise~\eqref{log-mal} would hold with $E_p$ replaced by $E'_p$ (and the same constants), leading to 
\begin{equation}
\sup_{\mathcal{K}_p} |\nabla_{t,x} v_\mu| \leq C \bigl( a\bigr)^{\alpha} \bigl( \sup_{\Omega_p} |\nabla_{t,x} v_\mu|\bigr)^{1- \alpha}
\Rightarrow \sup_{\mathcal{K}_p} |\nabla_{t,x} v_\mu| \leq C \epsilon {\sup_{\mathcal{K}_p} |\nabla_{t,x} v_\mu|},
\end{equation}
which, if $C\epsilon<1$ implies $v_{\mu} \mid_{\mathcal{K}_p}\equiv 0$ hence $v_\mu$ is identically $0$. 

As a consequence, 
$$ \int_{E_p} |v_\mu|^2 dx \geq \int_{E \setminus E'_p} |v_\mu|^2 dx \geq a^2 \frac \delta 2  \geq \frac \delta 2  \Bigl(  \epsilon \frac{\sup_{\mathcal{K}_p} |\nabla_{t,x} v_{\mu}|}{ \bigl( \sup_{\Omega_p} |\nabla_{t,x} v_{\mu} |\bigr)^{1- \alpha}} \Bigr) ^{\frac 2 \alpha}
$$ 
and we deduce

\begin{equation}\label{eq1.5}
\| u\|_{L^2(B(p,R)) } \leq C \bigl(\sup_{\Omega_p} |\nabla_{t,x} v_{\mu} |\bigr)^{1- \alpha} \| u \|_{L^2(E \cap B(p,R))}^{\alpha}.
\end{equation}
We now need  a variant of  Sobolev embeddings, which we prove for the reader's convenience: 
\begin{lem}\label{sobolev}
There exist $C, K>0$ such that for all $p\in \mathbb{Z}^d$,
and all $u\in L^2( \mathbb{R}^d;\kappa(x) dx )$,

$$ \Bigl(\sum_{p\in \mathbb{Z}^d} \| \nabla_{t,x} v_\mu \|^2_{L^\infty\bigl(( - T_2, T_2) \times B(p, 2R)\bigr)}\Bigr)^{1/2} \leq C e^{K\mu} \| u \|_{L^2( \mathbb{R}^d; \kappa(x) dx)}.$$
\end{lem}
{
\begin{rem} 
For smooth metrics and compact manifolds (without boundary), the spaces $\mathcal{H}^\sigma = D(-\Delta)^{\frac{\sigma}{2}}$ coincide with the usual Sobolev spaces $H^\sigma$, and Lemma~\ref{sobolev} is just the usual Sobolev injection (see below). At our level of regularity this is no longer the case.
\end{rem}}
Let us first, assuming Lemma~\ref{sobolev}, finish the proof of (the first part of) Theorem~\ref{princ}. Observe that Young's inequality yields 
$$a^{2- 2\alpha} b^{2\alpha} \leq C(a^2 + b^2),$$
Hence, fom~\eqref{eq1.5} we deduce for all $D>0$,
\begin{equation}\label{eq1.6}
\| u\|^2_{L^2(B(p,R)) } \leq C e^{-D\mu} \sup_{\Omega_p} |\nabla_{t,x} v_{\mu}|^2  + C e^{\frac{ 2- 2\alpha } {2\alpha} D \mu}\| u \|^2_{L^2(E \cap B(p,R))}.
\end{equation} Now, from Lemma~\ref{sobolev} we deduce
\begin{multline}\label{eq1.7}
\| u\|^2_{L^2(\mathbb{R}^d) }  \leq C e^{-D\mu} \sum_{p \in \mathbb{Z}}\sup_{\Omega_p} |\nabla_{t,x} v_{\mu}|^2  + C e^{D\frac{ 2- 2\alpha } {2\alpha} \mu}\sum_{p\in \mathbb{Z}}\| u \|^2_{L^2(E \cap B(p,R))}\\
\leq C e^{-D\mu}  e^{2K\mu} \| u \|^2_{L^2( \mathbb{R}^d)} + C e^{D\frac{ 2- 2\alpha } {2\alpha} \lambda}\| u \|^2_{L^2(E )},
\end{multline}
and the proof of Theorem~\ref{princ} for $\mu \geq  \mu_0$ follows from taking $D>2K$. \par 

Let us now prove Theorem~\ref{princ} in the second case, f $\mu\geq \mu_0$. 
From assumption~\ref{mesurebis}, we have 
 $$\mathcal{C}^{n} (E_p) \geq \delta, $$ and we deduce from ~\cite[Theorem 5.1]{LoMa17} that if $n$ is sufficiently close to $d$ that 

\begin{equation}\label{log-malbis} 
\sup_{\mathcal{K}_p} |\nabla_{t,x} v_\mu| \leq C \bigl( \sup_{E_p} |\nabla_{t,x} v_\mu| \bigr)^{\alpha} \bigl( \sup_{\Omega_p} |\nabla_{t,x} v_\mu|\bigr)^{1- \alpha},
\end{equation}
where the constant is uniform with respect to $p$  because the metric and $\kappa$ are bounded on $B(p, 2R)$ (uniformly with respect to $p$  while the metric is uniformly elliptic). We deduce
$$ \| u \| _{L^2( B(p, R)} \leq C\sup_{\mathcal{K}_p} |\nabla_{t,x} v_\mu| \leq C \bigl( \sup_{E_p} |\nabla_{t,x} v_\mu| \bigr)^{\alpha} \bigl( \sup_{\Omega_p} |\nabla_{t,x} v_\mu|\bigr)^{1- \alpha},$$
we deduce for all $D>0$,
\begin{equation}\label{eq1.6bis}
\| u\|^2_{L^2(B(p,R)) } \leq C e^{-D\mu} \sup_{\Omega_p} |\nabla_{t,x} v|^2  + C e^{\frac{ 2- 2\alpha } {2\alpha} D \mu}\| u \|^2_{L^\infty(E \cap B(p,R))}.
\end{equation}
Now, from Lemma~\ref{sobolev} we deduce
\begin{multline}\label{eq1.7bis}
\| u\|^2_{L^2(\mathbb{R}^d) }  \leq C e^{-D\mu} \sum_{p \in \mathbb{Z}}\sup_{\Omega_p} |\nabla_{t,x} v|^2  + C e^{D\frac{ 2- 2\alpha } {2\alpha} \mu}\sum_{p\in \mathbb{Z}}\| u \|^2_{L^\infty((E \cap B(p,R))\mathbb{R}^d))}\\
\leq C e^{-D\mu}  e^{2K\mu} \| u \|^2_{L^2( \mathbb{R}^d)} + C e^{D\frac{ 2- 2\alpha } {2\alpha} \lambda}\sum_{p\in \mathbb{Z}^d} \| u \|^2_{L^\infty(E \cap B(p,R))},
\end{multline}
and the proof of the second part in Theorem~\ref{princ} for $\mu \geq  \mu_0$ follows from taking $D>2K$. Finally, for $\mu< \mu_0$, apply Theorem~\ref{princ} for $\mu = \mu_0$ to 
$$ u = \Pi_\mu(u).$$
\begin{proof}[Proof of Lemma~\ref{sobolev} for smooth metrics] We start with an elementary proof (which works only for smooth metrics) relying only on Sobolev embeddings. Later on we give the general proof for Lipschitz metrics. Let $\chi \in C^\infty (\mathbb{R}^{d+1})$ non negative be equal to $1$ on $(- T_2, T_2) \times B(0, R)$. Let $s> \frac {d+1} 2$, and 
$$ \chi_p (x) = \chi (x-p). 
$$  By Sobolev embeddings, for all $p \in \mathbb{Z}^d$,
$$ \| \nabla_{t,x} v_\mu \|_{L^\infty(( - T_2, T_2) \times B(p, 2R))}\leq C \|  (\chi_p v_\mu)\| _{H^{s+1}},$$
and we deduce 
\begin{multline}
 \Bigl(\sum_{p\in \mathbb{Z}^d} \| \nabla_{t,x} v_\mu \|^2_{L^\infty( - T_2, T_2) \times B(p, 2R)}\Bigr)^{1/2} \leq C  \| v_\mu \|_{H^{s+1} (( -T_3, T_3) \times \mathbb{R}^d)}\\
  \leq \sum_{n=0}^{s+1}\| \partial_t ^n v_\mu \|_{L^2(( -T_3, T_3); H^{s+1-n}( \mathbb{R}^d))}. 
  \end{multline} We now use the elliptic regularity result (here we have to assume that $\kappa, g \in C^{k}, k \geq s$), 
$$\| w \| _{H^{\sigma}( \mathbb{R}^d)} \sim \|( ( - \Delta)^{\sigma/2} + 1 )w \| _{L^2( \mathbb{R}^d)}
$$ 
and we get (recal that $- \Delta$ acts on the spectral measure as the multiplication by $\lambda^2$):
\begin{multline}
 \Bigl(\sum_{p\in \mathbb{Z}^d} \| \nabla_{t,x} v_\mu \|^2_{L^\infty( - T_2, T_2) \times B(p, 2R)}\Bigr)^{1/2}   \leq C\sum_{n=0}^{s+1}\| \partial_t ^n v_\mu \|_{L^2(( -T_3, T_3); H^{s-n}( \mathbb{R}^d, \kappa(x) dx))} \\
 \leq C \sum_{n=0}^s\| \int_{{\lambda=0}} ^\mu \frac{ \frac{ d^n} {dt^n} (\sinh) ( \lambda t)}{ \lambda} ( \lambda ^ {s+1-n} + 1) d\mu_\lambda u \|_{L^2(( -T_3, T_3); L^2(\mathbb{R}^d, \kappa (x) dx))} 
 \leq C e^{C\mu} \| u \|_{L^2( \mathbb{R}^d)}
 \end{multline}
 where in the last estimate we used~\eqref{bound}. 
\end{proof}

\begin{proof}[Proof of Lemma~\ref{sobolev} for Lipschitz metrics] Let us now come back to the case of Lipschitz metrics.
 We shall use the following results from~\cite{HaLi97} about weak solutions to 
 \begin{equation}\label{weak}
  -\sum_{i,j} \partial_{y_i} a_{i,j} \partial_{y_j} w =f, \qquad  \textrm{in } B(0,1),
  \end{equation}  with $B(0,1) \subset \mathbb{R}^n$ and 
  $$ \lambda |\xi|^2 \leq \sum_{i,j =1}^n a_{i,j}(x) \xi_i \xi_j \leq \Lambda |\xi|^2, \qquad x \in B(0,1), \quad \xi \in \mathbb{R}^n. $$
   
 \begin{thm}[\protect{\cite[Theorem 3.13, combined with Theorem 3.1]{HaLi97}}]Consider $w \in H^1( B(0,1))$ a weak solution to~\eqref{weak}. Assume that $a_{i,j}\in C^{0,\alpha}( \overline{B(0,1)})$,  $f \in L^q(B(0,1))$, for some $q >n$ and $\alpha = 1 - \frac n q \in (0,1)$. Then $\nabla_y w \in C^{0,\alpha}(B(0,1))$. Moreover there exists $M = M( n, \lambda, \| a_{i,j} \|_{C^{0,\alpha}})>0$ such that we have the estimate
  \begin{multline}
 \|\nabla_yw\|_{C^{0,\alpha}(B(0, \frac 1 2))} = \sup_{B(0, \frac 1 2)} |\nabla_y w| +  \sup_{\smash{x,y \in B(0, \frac 1 2), \, x \not = y}} \frac{|\nabla_y w(y) - \nabla_y w(y')|}{ |y - y'|^{\alpha} }   \\
   \leq M (\| f\|_{L^q(B(0,1))}  +  \|w \|_{H^1(B(0,1))}).
 \label{eq:gradient estimate Holder}
 \end{multline}  
   \end{thm}

 Now, we have on $\mathbb{R}_t \times \mathbb{R}^{d}_x = \mathbb{R}^{n}_y$, 
 \begin{gather}
  (- T_2, T_2) \times B(0,2R) \subset B \left(  (0,0), \sqrt{4R^2+T_2^2}   \right), \\
   B \left( (0,0), 2 \sqrt{ 4R^2+ T_2^2}  \right) \subset \left( -2\sqrt{ 4R^2+ T_2^2}, 2 \sqrt{ 4R^2+ T_2^2} \right) \times B \left( 0, 2\sqrt{ 4R^2+ T_2^2}  \right),
   \end{gather}
 and considering the function 
 $$ \widetilde{w} (t,x) =  v_{\mu} \left(2t\sqrt{4R^2+T_2^2} , 2x\sqrt{4R^2+T_2^2} \right), \qquad (t,x) \in B((0,0), 1)$$ which satisfies 
 $$ \sum_{p, q=1}^d \partial_{x_p}a_{p,q} (t,x) \partial_{x_q} \widetilde{w}  + \partial_t b(x) \partial _t \widetilde{w} =0,
 $$
 $$ a_{p,q}(x) = g^{-1} (\tau x) \kappa (\tau x ), \qquad b(x) = \kappa (\tau x), \qquad \tau = 2\sqrt{4R^2+T_2^2}
 $$
  we get from~\eqref{eq:gradient estimate Holder} (with $f=0$):
 \begin{align*}
 \sup_{(- T_2, T_2) \times B(0, 2R)}  |\nabla_{t,x} v_{\mu}|  & \leq \|\nabla_{t,x} \tilde{w} \|_{C^{0,\alpha}(B(0, \frac 1 2))} \\ 
  & \leq  M  \|  \tilde{w} \|_{H^1(B(0,1))} \\
  & \leq M \| v_{\mu} \|_{H^1 \left(( -2\tau, 2 \tau ) \times B(0, 2\tau) \right)},
   \end{align*} where the constant $M$ depends only on the constants $A,a$ in assumption~\eqref{ellipticity},  the dimension of space, $d$ , $T_2$ and $R$.
 Since the bounds in~\eqref{ellipticity} are translation invariant, we get for all $p \in \mathbb{Z}^d$, {\em with the same constant~$M$}, 
 \begin{equation*}
 \sup_{(- T_2, T_2) \times B(p, 2R)}  |\nabla_{t,x} v_{\mu}| 
   \leq M \| v_{\mu} \|_{H^1( \left( -2\tau, 2 \tau ) \times B(p, 2\tau \right)}.
   \end{equation*}
  As a consequence, summing with respect to $p$, using quasi-orthogonality for the $H^1$ norm, we get
  \begin{equation*}
 \sum_{p \in \mathbb{Z}^d} \sup_{(- T_2, T_2) \times B(p, 2R)}  |\nabla_{t,x} v_{\mu}|^2
   \leq C \|  v_{\mu} \|^2_{H^1 \left( ( -\tau, \tau ) \times \mathbb{R}^d \right)}.
   \end{equation*} On the other hand, using (\ref{bound}) we have 
   \begin{align*}
    \| \partial_tv_\mu\| ^2_{L^2( -\tau,\tau ) \times \mathbb{R}^d;  \kappa dx dt)} & = \left\| \left\| \int_0 ^\mu \cosh  ( \lambda t) d\mu_{\lambda} u \right\|_{L^2( \kappa dx)}^2    \right\|^2_{L^2_t}  \\
   &  \leq \left\|  \left(  \max_{0\leq \lambda \leq \mu }  \cosh  ( \lambda t)   \right)^2   \left\|  u \right\|_{L^2( \kappa dx)}^2    \right\|^2_{L^2_t}  \\
    &  \leq \int_{-\tau}^\tau \cosh ^2 ( \mu t) d t \| u\|_{L^2_x}^2  \leq C e^{2T \mu} \| u\|_{L^2(\kappa dx)}^2.
   \end{align*}
   Next, using that the metric $g$ is uniformly elliptic according to (\ref{ellipticity}), for any $t\in (-\tau,\tau)$, we have  
   \begin{align*}
    \| \nabla_x v_{\mu}(t) \|^2_{L^2(\kappa(x) dx)} & = \int_{\R^d} |\nabla_x v_{\mu}|^2 (t,x)  \kappa(x) dx \\
    & \leq C(a,A) \int_{\R^d} \sum_{i,j} g^{-1}_{i,j} (x)  \partial_{x_i} v_{\mu}(t,x) \partial_{x_j} v_{\mu}(t,x) \kappa (x) dx \\
    & = C(a,A) \bigl( - \Delta v_{\mu}(t), v_{\mu}(t) \bigr)_{L^2 ( \kappa(x) dx)}.
   \end{align*} Using again (\ref{bound}) we deduce
\begin{multline*}
    \| \nabla_x v_\mu\| ^2_{L^2((-\tau,\tau) \times \mathbb{R}^d); \kappa dx dt} \leq C \left\|  \bigl( -\Delta v_\mu, v_{\mu} \bigr)_{L^2(\kappa (x) dx)} \right\| ^2_{L^2_t} \\
  = \int_{-\tau}^\tau   \left(  \max_{0\leq \lambda \leq \mu }   \lambda \sinh  ( \lambda t)   \right)^2  dt  \, \, \| u\|_{L^2(\kappa(x) \ud x)}^2 \\
  \leq C e^{2\tau \mu} \| u\|_{L^2(\kappa(x) \ud x)}^2.
   \end{multline*} This concludes the proof of Proof of Lemma~\ref{sobolev} for Lipschitz metrics.
\end{proof}

\end{document}